\documentclass[english]{amsart}

\usepackage{amsmath}
\usepackage{dsfont}
\usepackage{amssymb}
\usepackage{amsthm}
\usepackage{amscd}
\usepackage{babel}
\usepackage[latin1]{inputenc}
\usepackage{graphicx}

\newtheorem{Theorem}{Theorem}

\newtheorem{Corollary}{Corollary}

\newtheorem{Proposition}{Proposition}

\newtheorem{Lemma}[Theorem]{Lemma}

\theoremstyle{definition}
\newtheorem{rem}{Remark}

\renewcommand{\subjclassname}{AMS \textup{2010} Mathematics Subject
Classification\ }

\author{Jos\'{e} Mar\'{i}a Grau}
\address{Departamento de Matem\'{a}ticas, Universidad de Oviedo\\ Avda. Calvo Sotelo s/n, 33007 Oviedo, Spain}
\email{grau@uniovi.es}

\author{C. Miguel}
\address{Instituto de Telecomunica\c c\~oes, Beira Interior University, Department of Mathematics, Covilh\~a, Portugal}
\email{celino@ubi.pt}

\author{Antonio M. Oller-Marc\'{e}n}
\address{Centro Universitario de la Defensa de Zaragoza\\ Ctra. Huesca s/n, 50090 Zaragoza, Spain}
\email{oller@unizar.es}

\title{On the zero divisor graphs of the ring of Lipschitz integers modulo $n$.}

\begin{document}

\begin{abstract} This article studies the zero divisor graphs of the ring of Lipschitz integers modulo $n$.  In particular we focus on the number of vertices, the diameter and the girth. We also give some results regarding the domination number of these graphs.

\end{abstract}

\maketitle
\subjclassname{: 05C75; 11R52.}


\section{Introduction}
The idea of studying the interplay between ring-theoretic properties of a ring $R$ and graph-theoretic properties of a graph defined after it, is quite recent. It was first introduced for commutative rings by Beck in 1988 \cite{be}. In Beck's definition the vertices of the graph are the elements of the ring and two distinct  vertices $x$ and $y$ are adjacent if and only if $xy=0$. Later, Anderson and Livingston \cite{an} slightly modified this idea, considering only the non-zero zero divisors of $R$ as vertices of the graph with the same adjacency condition. Redmond \cite{re} extended this notion of zero-divisor graph to noncommutative rings.

Given a noncommutative ring $R$, we define two different graphs associated to $R$. The directed zero divisor graph, $\Gamma(R)$, and the undirected zero divisor graph, $\overline\Gamma(R)$. Both graphs share the same vertex set, namely, the set $Z(R)^{*}$ of non-zero zero divisors of $R$. In $\Gamma(R)$, given two distinct vertices $x$ and $y$, there is a directed edge of the form $x\rightarrow y$ if and only if $xy=0$. On the other hand, two distinct vertices $x$ and $y$ of $\overline\Gamma(R)$ are connected by an edge if and only if either $xy=0$ or $yx=0$.
Several properties of zero divisor graphs of different general classes of rings are studied in \cite{ak, ak1, ku, re, wu}.

Recall that, if $n>1$ is a rational integer and $\langle n\rangle$ is the ideal in the Gaussian integers generated by $n$, then the factor ring $\mathbb{Z}[i]/\langle n\rangle$ is isomorphic to the Gaussian integers modulo $n$
\begin{equation*}\mathbb{Z}_n[i]:=\{a+bi : a, b\in\mathbb Z_n\}.\end{equation*}
The zero divisor graph of the ring of Gaussian integers modulo $n$ has recently received great attention \cite{ab, ab1, na}.

The algebraic construction defined above for the Gaussian integers can be easily extended to the ring $\mathbb{Z}[i,j,k]$ of Lipschitz integer quaternions. Indeed, let
again be $n>1$ a rational integer and denote by $\langle n\rangle$ the principal ideal in $\mathds Z[i, j, k]$ generated by $n$. Then, the factor ring $\mathds Z[i, j, k]/\langle n\rangle$ is isomorphic to
\begin{equation*}\mathds Z_n[i, j, k]:=\{a+ib+cj+dk:a, b, c, d\in\mathbb{Z}_n\},\end{equation*}
which is called the ring of Lipschitz quaternions modulo $n$.

The aim of this paper is to study the zero divisor graphs of the ring of Lipschitz quaternions modulo $n$, both the directed $\Gamma(\mathbb Z_n[i, j, k])$ and the undirected $\overline\Gamma(\mathbb Z_n[i, j, k])$.

If $n=p_1^{r_1}\ldots p_k^{r_k}$, is the prime power factorization of $n$, the Chinese remainder theorem induces a natural isomorphism
\begin{equation}\label{decom}\mathbb Z_n[i, j, k]\cong\mathbb Z_{p_1^{r_1}}[i, j, k]\oplus\ldots\oplus\mathbb Z_{p_k^{r_k}}[i, j, k].\end{equation}
Therefore, in order to study the structure of the rings $\mathbb Z_n[i, j, k]$ we can restrict ourselves to the prime power case.

If $p$ is an odd prime and $l$ is a positive integer, then  $\mathbb Z_{p^l}[i, j, k]$ is isomorphic to the full matrix ring $M_2(\mathbb Z_{p^l})$ \cite{gmo,vi}. Consequently, for an odd positive integer $n$ the ring $\mathbb Z_n[i, j, k]$ is isomorphic to the matrix ring $M_2(\mathbb Z_{n})$. Hence, in this case we can use known results about the zero divisor graph of matrix rings over commutative rings.

Unfortunately, if $n$ is even it is no longer true that $\mathbb Z_n[i, j, k]$ is isomorphic to the matrix ring $M_2(\mathbb Z_n)$. In fact, note that an element
$z=z_0+z_1i+z_2j+z_3k\in\mathbb Z_n[i, j, k]$ is a unit if and only if its norm $\|z\|=z_0^2+z_1^2+z_2^2+z_3^2$ is a unit in $\mathbb Z_n$. Since
\begin{equation*}\label{E1}\|z+w\|=\|z\|+\|w\|+2\textrm{Re}(z\overline{w}),\end{equation*}
the sum of two units of $\mathbb Z_n[i, j, k]$ is never a unit. This fact is clearly false in $M_2(\mathbb Z_n)$ and the claim holds.

\section{The number of vertices}

Recall that both graphs $\Gamma(\mathbb Z_n[i, j, k])$ and $\overline\Gamma(\mathbb Z_n[i, j, k])$ share the same set of vertices. Namely, the non-zero zero divisors of $\mathbb Z_n[i, j, k]$. Due to the isomorphism (\ref{decom}) we can focus on the case when $n$ is a prime power.

For an odd prime power $p^\alpha$, we have the isomorphism $\mathbb Z_{p^\alpha}[i, j, k]\cong M_2(\mathbb Z_{p^\alpha})$. Thus, it is enough to determine the number of non-zero zero divisors in the matrix ring $M_2(\mathbb Z_{p^{\alpha}})$. We do so in the following proposition.

\begin{Proposition}
Let $p$ be an odd prime number and $\alpha\geq 1$. Then, the
number of zero divisors in $M_2(\mathbb{Z}_{p^\alpha})$ is
$$p^{4\alpha-1}+p^{4\alpha-2}-p^{4\alpha-3}.$$
\end{Proposition}
\begin{proof}
Let us denote by $\mathcal{N}(M_2(\mathbb{Z}_{p^{\alpha}}))$ the nilradical of $M_2(\mathbb{Z}_{p^{\alpha}})$. It is known \cite[p.125]{ke} that $\mathcal{N}(M_2(\mathbb{Z}_{p^{\alpha}}))=M_2(\mathcal{N}(\mathbb{Z}_{p^{\alpha}}))$ and recall that $\mathcal{N}(\mathbb{Z}_{p^{\alpha}})=\{a\in\mathbb{Z}_{p^{\alpha}}:\gcd(a,p)>1\}$ is precisely the unique maximal ideal of the local ring $\mathbb{Z}_{p^{\alpha}}$.
This implies that the factor ring $M_2(\mathbb{Z}_{p^\alpha})/\mathcal{N}(M_2(\mathbb{Z}_{p^\alpha}))$ is simple, because every ideal of the full matrix ring $M_n(R)$ over a unital commutative ring $R$ is of the form $M_n(L)$ for some ideal $L$ of $R$.

Now, since every artinian simple ring is isomorphic to the full matrix ring over a division ring \cite[p.39]{lam} and by the well-known Wedderburn's little theorem a finite division ring is a field, we conclude that
$M_2(\mathbb{Z}_{p^\alpha})/\mathcal{N}(M_2(\mathbb{Z}_{p^\alpha}))$ is isomorphic to a full matrix ring over a finite field.

Moreover, since
\begin{equation*}\left|M_2(\mathbb{Z}_{p^\alpha})/\mathcal{N}(M_2(\mathbb{Z}_{p^\alpha}))\right|=\frac{p^{4\alpha}}{p^{4(\alpha-1)}}=p^4,\end{equation*}
it follows that $M_2(\mathbb{Z}_{p^\alpha})/\mathcal{N}(M_2(\mathbb{Z}_{p^\alpha}))$ is isomorphic to the full matrix ring $M_2(\mathbb{Z}_p)$. 

Finally, since every element of $\mathcal{N}(M_2(\mathbb{Z}_{p^\alpha}))$ is a zero divisor and
there are exactly $p^3+p^2-p$ zero divisors in $M_2(\mathbb{Z}_p)$, it follows that the number of zero divisors in
$M_2(\mathbb{Z}_{p^\alpha})$ is
$$p^{4(\alpha-1)}(p^3+p^2-p)=p^{4\alpha-1}+p^{4\alpha-2}-p^{4\alpha-3},$$
as claimed.
\end{proof}

\begin{rem}
\label{remuniimp}
Since in a finite unital ring every element is either a zero divisor or a unit, it follows that the number of units in $M_2(\mathbb Z_{p^\alpha})$ for an odd prime $p$ is
\begin{equation*}\label{uniimpar}p^{4\alpha}-p^{4\alpha-1}-p^{4\alpha-2}+p^{4\alpha-3}.\end{equation*}
\end{rem}

If $p=2$ the isomorphism $\mathbb Z_{2^t}[i, j, k]\cong M_2(\mathbb Z_{2^t})$ is no longer true. Consequently, we must use a different approach in order to compute the number of non-zero zero divisors of $\mathbb Z_{2^t}[i, j, k]$.

\begin{Proposition}
\label{t1}
The number of vertices of the graph $\overline\Gamma(\mathbb Z_{2^t}[i, j , k])$ is  $2^{4t-1}-1$.
\end{Proposition}
\begin{proof}
Note that an element $z=x_0+x_1i+x_2j+x_3k\in\mathbb Z_{2^t}[i, j , k]$ is a unit if and only if its norm
 $\|x\|=x_0^2+x_1^2+x_2^2+x_3^2$ is a unit in $\mathbb Z_{2^t}$. On the other hand, an element is a unit in $\mathbb Z_{2^t}$ if its reduction modulo $2$ is a unit. So consider the composition of homomorphisms
 \begin{equation*}\mathbb Z_{2^t}[i, j , k] \xrightarrow{\|.\|}\mathbb Z_{2^t}\xrightarrow{mod\; 2}\mathbb Z_2,\end{equation*}
 the kernel of this composition if the set of zero divisors in $\mathbb Z_{2^t}[i, j , k]$. Thus,
 \begin{equation*}
\mathbb Z_{2^t}[i, j , k]/ Z(\mathbb Z_{2^t}[i, j , k])\cong \mathbb Z_2,
\end{equation*}
 where  $Z(\mathbb Z_{2^t}[i, j , k])$ denotes the set of zero divisors of $\mathbb Z_{2^t}[i, j , k]$. Therefore $|\mathcal U(\mathbb Z_{2^t}[i, j , k])|=| Z(\mathbb Z_{2^t}[i, j , k])|=2^{4t-1}.$ Hence, there are exactly $2^{4t-1}-1$ non-zero zero divisors in $\mathbb Z_{2^t}[i, j , k]$, which is the number of vertices of the graph $\overline\Gamma(\mathbb Z_{2^t}[i, j , k])$ as claimed.
\end{proof}

\begin{rem}
\label{remuni2}
Observe that,  as in Remark \ref{remuniimp},  the number of units in $\mathbb Z_{2^t}[i, j , k]$ is $2^{4t-1}$.
\end{rem}

Recall that, given a direct sum of rings $R=R_1\oplus\ldots\oplus R_k$, an element $r\in R$ is a unit if and only if every projection of $r$ in $R_i$ is a unit in $R_i$. Hence, if $n=2^tp_1^{\alpha_1}\cdots p_k^{\alpha_k}$ is the prime power decomposition of $n$, isomorphism (\ref{decom}) together with Remarks \ref{remuniimp} and \ref{remuni2} lead to:
\begin{equation*}
|\mathcal U(\mathbb Z_{n}[i, j , k])|=\begin{cases}
\prod\limits_{i=1}^k\left(p_i^{4\alpha_i}-p_i^{4\alpha_i-1}-p_i^{4\alpha_i-2}+p_i^{4\alpha_i-3}\right), & \textrm{if $t=0$};\\ 2^{4t-1}\prod\limits_{i=1}^k(p_i^{4\alpha_i}-p_i^{4\alpha_i-1}-p_i^{4\alpha_i-2}+p_i^{4\alpha_i-3}), & \textrm{if $t>0$}.
\end{cases}
\end{equation*}

As a consequence of the previous work we have the main result of this section.

\begin{Theorem}
Let $n=2^tp_1^{\alpha_1}\cdots p_k^{\alpha_k}$ be the prime power decomposition of $n$. Then, the number of vertices in
the graph $\Gamma(\mathbb Z_n[i, j, k])$ or $\overline\Gamma(\mathbb Z_n[i, j, k])$ is:
$$
|V(\Gamma(\mathbb Z_n[i, j, k]))|=\begin{cases}
n^4-\prod\limits_{i=1}^k\left(p_i^{4\alpha_i}-p_i^{4\alpha_i-1}-p_i^{4\alpha_i-2}+p_i^{4\alpha_i-3}\right)-1, & \textrm{if $t=0$};\\
n^4-2^{4t-1}\prod\limits_{i=1}^k(p_i^{4\alpha_i}-p_i^{4\alpha_i-1}-p_i^{4\alpha_i-2}+p_i^{4\alpha_i-3})-1, & \textrm{if $t>0$}.
\end{cases}
$$
\end{Theorem}
\begin{proof}
Just apply the previous observation recalling that the non-zero elements of a finite ring are either units or zero-divisors.
\end{proof}

\section{The diameter}

We recall that the {\it distance} between two distinct vertices $a$ and $b$ of a graph, denoted by $\textrm{d}(a,b)$, is the length of the shortest path connecting them (the distance being infinity if no such path exists). The {\it diameter} of a graph $G$, denoted by $\textrm{diam}(G)$, is given by
\begin{equation*}\textrm{diam}(G)=sup\{\textrm{d}(a, b) : \textrm{$a, b$ distinct vertices of $G$}\}.\end{equation*}

Our objective in this section is to find the diameter of the directed zero divisor graph $\Gamma(\mathbb Z_n[i, j, k])$ and of the undirected zero divisor graph $\overline\Gamma(\mathbb Z_{n}[i, j , k])$.

Recall that $Z_L(R)$ and $Z_R(R)$ denote, respectively, the set of left and right zero divisors of $R$. The following result was proved in \cite{re}.

\begin{Theorem}
\label{rediam}
Let $R$ be a noncommutative ring, with $Z^{*}(R)\neq\emptyset$. Then $\Gamma(R)$ is connected if and only if $Z_L(R)=Z_R(R)$. If $\Gamma(R)$ is connected, then $\rm{diam}(\Gamma(R))\leq 3$.
\end{Theorem}

Note that in any finite ring $R$ we have $Z_L(R)=Z_R(R)$. Hence, the previous theorem implies that the directed zero divisor graph $\Gamma(\mathbb Z_n[i, j, k])$ is connected and $diam(\Gamma(\mathbb Z_n[i, j, k]))\leq 3$. We will now see that, in many cases, equality holds. To do so, we first need a technical result involving the direct sum of finite unital noncommutative rings. A commutative version was established in \cite{ax}.

\begin{Lemma}
\label{l3}
  
Let $R=R_1\oplus R_2$, where $R_1$ and $R_2$ are finite unital noncommutative rings.  Then. ${\rm diam}(\Gamma(R))=3$.
\end{Lemma}
\begin{proof}
First note that, since  $R_1$ and $R_2$ are finite and noncommutaive if follows by the Wedderburn's little theorem that both $R_1$ and $R_2$ have nonzero zero divisors.

On the other hand, since $R_1$ and $R_2$ are rings with identity we can choose a unit $u_1$ from $R_1$ and a unit $u_2$ from $R_2$. Let
$x\in Z(R_1)^\star$ and $y\in Z(R_2)$ and consider the elements $(x, u_2), (u_1, y)\in Z(R)^\star$. We will prove that the distance between the vertices $(x, u_2)$ and  $(u_1, y)$ is 3. Indeed $(x, u_2)(u_1, y)=(xu_1, u_2y)\neq (0, 0)$. Hence $d((x, u_2),(u_1, y))>1$. On the other hand if $(a, b)\in Z(R)^\star$ satisfies
\begin{equation*}(x, u_2)(a, b)=(a, b)(u_1, y)=(0, 0),\end{equation*}
then we have $u_2b=0$ and $au_1=0$ implying $a=b=0$, a contradiction. Therefore, $d((x, u_2),(u_1, y))>2$. Finally, using theorem \ref{rediam} we get the result.$\Box$
\end{proof}

As a consequence, we have the following result.

\begin{Proposition}
Let $n$ be an integer divisible by at least two primes. Then,
$$\textrm{diam}(\Gamma(\mathbb{Z}_n[i,j,k]))= 3.$$
\end{Proposition}
\begin{proof}
It is enough to apply isomorphism (\ref{decom}) together with Lemma \ref{l3}.
\end{proof}

\begin{rem}
It is clear that $\textrm{diam}(\overline\Gamma(\mathbb{Z}_n[i,j,k]))\leq\textrm{diam}(\Gamma(\mathbb{Z}_n[i,j,k]))$. Now, if $n$ is divisible by at least two primes, there exist vertices in $\overline\Gamma(\mathbb{Z}_{n}[i,j,k])$ that are not at distance 2. Hence, we obtain the equality $\textrm{diam}(\overline\Gamma(\mathbb{Z}_n[i,j,k]))=\textrm{diam}(\Gamma(\mathbb{Z}_n[i,j,k]))$ in this case.
\end{rem}

Now, we must focus on the prime power case. We first look at the odd case, where the following technical lemma is useful \cite[Lem. 4.2; Cor. 4.1]{bo}.

\begin{Lemma}
\label{l1}
Let $R$ be a commutative ring. If every finite set of zero divisors from $R$ has a non-zero annihilator, then ${\rm diam}(\Gamma(M_n(R))=2$. In particular, if $F$ is a field, then ${\rm diam}(\Gamma(M_n(F))=2$.
\end{Lemma}

\begin{Proposition}
Let $t\geq 1$ and let $p$ be an odd prime. Then,
$$\textrm{diam}(\Gamma(\mathbb{Z}_{p^t}[i,j,k]))=2.$$
\end{Proposition}
\begin{proof}
If $t=1$, $\mathbb Z_p$ is a field and the result follows from the second part Lemma \ref{l1}. Now assume that $t>1$. In this case the maximal ideal of the local ring $\mathbb Z_{p^t}$ is the principal ideal generated by $p$. This maximal ideal is nilpotent with index of nilpotence $t$. Therefore, the element $p^{t-1}$ belongs to the annihilator of every zero divisor in $\mathbb Z_{p^l}$ and Lemma \ref{l1} applies again.
\end{proof}

\begin{rem}
Again, $\textrm{diam}(\overline\Gamma(\mathbb{Z}_{p^t}[i,j,k]))\leq\textrm{diam}(\Gamma(\mathbb{Z}_{p^t}[i,j,k]))$. Now, if $p$ is an odd prime, there exist vertices in $\overline\Gamma(\mathbb{Z}_{p^t}[i,j,k])$ that are not at distance 1. Hence, we obtain the equality $\textrm{diam}(\overline\Gamma(\mathbb{Z}_{p^t}[i,j,k]))=\textrm{diam}(\Gamma(\mathbb{Z}_{p^t}[i,j,k]))$ in this case.
\end{rem}

Finally, we turn to the $p=2$ case. The following result computes the diameter of the undirected zero divisor graph.

\begin{Proposition}
\label{t2}
Let $t\geq 1$. Then $\textrm{diam}(\overline\Gamma(\mathbb Z_{2^t}[i, j , k]))=2$.
\end{Proposition}
\begin{proof}
The graph $(\overline\Gamma(\mathbb Z_{2^t}[i, j , k])$ is not complete since the vertices $1+i$ and $1+j$  are not adjacent.
On the other hand, from equation (\ref{E1}) if follows that the sum of two zero divisors in $\mathbb{Z}_{p^t}[i,j,k]$ is a zero divisor. Hence, the set of all zero divisors is an ideal. Clearly, this ideal equals the nilradical $\mathcal N(\mathbb Z_{2^t}[i, j , k])$.
 Since the nilradical is nilpotent, it follows that there is a vertex of $\overline\Gamma(\mathbb Z_{2^t}[i, j , k])$ adjacent to all others, and thus $\textrm{diam}(\overline\Gamma(\mathbb Z_{2^t}[i, j , k]))=2$.
\end{proof}

Recall that a directed graph $G$ is called symmetric if, for every directed edge $x\rightarrow y$ that belongs to $G$, the corresponding reversed edge $y\rightarrow x$ also belongs to $G$. We are going to prove that the directed zero divisor graph $\Gamma(\mathbb Z_{2^t}[i,j,k])$ is symmetric. As a consequence, it will follow that $\textrm{diam}(\overline\Gamma(\mathbb{Z}_{2^t}[i,j,k]))=\textrm{diam}(\Gamma(\mathbb{Z}_{2^t}[i,j,k]))$.

A ring $R$ is called reversible \cite{co} if, for every $a,b\in R$, $ab=0$ implies that $ba=0$. Clearly, a ring $R$ is reversible if and only if its directed zero divisor graph $\Gamma(R)$ is symmetric. Thus, to prove that $\Gamma(\mathbb Z_{2^t}[i, j , k])$ is symmetric we will prove that $\mathbb{Z}_{2^t}[i,j,k]$ is reversible. To do so, we need a series of technical lemmata.

\begin{Lemma}
\label{L4}
Let $w\in \mathbb{Z}[i, j, k]$. If $\|w\|\equiv 0 \pmod{4}$, then either all the components of $w$ are even or all of them are odd.
\end{Lemma}
\begin{proof}
Put $w=a_1+a_2i+a_3j+a_4k$ and denote by $n_w:=\textrm{card} \{i : a_i\ \textrm{is even}\}$. Since $a_i^2\equiv 0,1\pmod{4}$, it follows that $0\equiv||w||\equiv n_w\pmod{4}$ and hence the result.
\end{proof}

\begin{Lemma}
\label{L8}
Let $w\in \mathbb{Z}[i,j,k]$. If $\|w\|\equiv 0 \pmod {8}$, then all the components of $w$ are even.
\end{Lemma}
\begin{proof}
Put $w=a_1+a_2i+a_3j+a_4k$. Since $\|w\|=a_1^2+a_2^2+a_3^2+a_4^2\equiv 0\pmod{4}$ the previous lemma implies that either all the components of $w$ are even or all of them are odd. Assume that all of them are odd and put $a_i=2a_i'+1$ for every $i$. Then, $\|w\|=a_1^2+a_2^2+a_3^2+a_4^2=4((a_1')^2+a_1'+(a_2')^2+a_2'+(a_3')^2+a_3'+(a_4')^2+a_4)+4$. Hence, $((a_1')^2+a_1'+(a_2')^2+a_2'+(a_3')^2+a_3'+(a_4')^2+a_4')+1\equiv 0\pmod{2}$. This is clearly a contradiction and the result follows.
\end{proof}

\begin{Proposition}
\label{rev}
Let $w,z \in \mathbb{Z}[i,j,k]$. If $wz \equiv 0 \pmod {2^s} $, then $zw \equiv 0 \pmod {2^s}$. In other words, the ring $\mathbb Z_{2^t}[i, j,k]$ is reversible.
\end{Proposition}
\begin{proof}
We will proceed by induction on $s$.

The case $s=1$ is obvious since $\mathbb{Z}[i, j, k]/ 2 \mathbb{Z}[i, j,k]$ is trivially commutative.

Let us consider $s=2$ and assume that $w z \equiv 0 \pmod {4}$. Hence, $\|w\|\|z\|=\|wz\|\equiv 0 \pmod {16}$. If $\|w\|\equiv 0\mod{8}$ we can apply Lemma \ref{L8} to conclude that $w=2w'$ for some $w'\in\mathbb{L}$. Hence (using the case $s=1$) we have,
$$w z \equiv 0 \pmod {4} \Leftrightarrow w'z\equiv 0\pmod{2} \Leftrightarrow zw'\equiv 0\pmod{2} \Leftrightarrow zw\equiv 0\pmod{4}.$$
the same holds if $\|z\|\equiv 0 \pmod{8}$. Finally, if both $\|w\|,\|z\|\equiv 0\pmod{4}$ we apply Lemma \ref{L4} to conclude that all the components of $w$ and $z$ are odd. But in this case it can be easily seen that $zw-wz\in 4\mathbb{Z}[i, j,k]$ and the result follows.

Now, assume that $s>2$ and that $w z \equiv 0 \pmod {2^s}$. In this case $\|w\|\|z\|=\|wz\|\equiv 0 \pmod {2^{2s}}$ and, since $s>2$ it follows that either $\|w\|\equiv 0 \pmod{8}$ or $\|z\|\equiv 0 \pmod{8}$. If, for instance, $\|w\|\equiv 0 \pmod{8}$ we apply Lemma \ref{L8} again to conclude that $w=2w'$ for some $w'\in\mathbb{Z}[i, j ,k]$ and we can proceed like in the previous paragraph. The same holds if $\|z\|\equiv 0\pmod{8}$ and the proof is complete.
\end{proof}

\begin{rem}
The concept of {\em symmetric} ring was defined by Lambek in \cite{SIM}: a ring $R$ is symmetric if, for every $a,b,c \in R$, $abc=0$ implies that also $acb=0$. It has sometimes been erroneously asserted (and even ``proved'') that reversible and symmetric are equivalent conditions. If a unital ring is symmetric, then it is also reversible. But this is no longer true for non-unital rings, as illustrated by an example of Birkenmeier \cite{BIR}. In the case of unital rings, the smallest known reversible non-symmetric ring was given in \cite{REV2}. Namely, it is the group algebra $\mathbb{F}_2 Q_8$ where $Q_8$is the quaternion group. In \cite{GUT} it was proved that this is in fact the smallest reversible group algebra over a field which is not symmetric. In \cite{LI} it was also confirmed that $\mathbb{F}_2 Q_8$ is indeed the smallest reversible group ring which is not symmetric. Note that $\mathbb{Z}_4[i, j, k]$ is a reversible ring due to Proposition \ref{rev} which is trivially non-symmetric. In fact, it is the smallest known ring with characteristic different from 2 with this property, having the same number of elements (256) as the aforementioned example $\mathbb{F}_2 Q_8$.
\end{rem}

The reversibility of $\mathbb Z_{2^t}[i, j,k]$ implies that $\textrm{diam}(\overline\Gamma(\mathbb{Z}_{2^t}[i,j,k]))=\textrm{diam}(\Gamma(\mathbb{Z}_{2^t}[i,j,k]))$, so from all the previous work we obtain the following result.

\begin{Theorem}
\label{diam}
Let $n$ be any integer. Then
$${\rm diam}(\overline{\Gamma}(\mathbb Z_n[i, j, k]))={\rm diam}(\Gamma(\mathbb Z_n[i, j, k]))=\begin{cases} 2, & \textrm{if $n$ is a prime power};\\ 3, & \textrm{otherwise}.\end{cases}$$
\end{Theorem}

Recall that a graph $G$ is {\emph complete} provided every pair of distinct vertices is connected by a unique edge. In \cite[Theorem 15]{ab1} it was proved that the undirected zero divisor graph for the ring of Gaussian integers modulo $n$, $\overline{\Gamma}(\mathbb{Z}_n[i])$, is complete if and only if $n=q^2$, where $q$ is a rational prime such $q\equiv 3\pmod{4}$. In our case we have the following.

\begin{Corollary}
The graph $\overline{\Gamma}(\mathbb Z_n[i, j, k])$ is never complete.
\end{Corollary}
\begin{proof}
The diameter of a complete graph is 1. Since this is not possible due to Theorem \ref{diam}, the result follows.
\end{proof}

\section{The girth}
A {\it cycle} in a graph is a path starting and ending at the same vertex. The {\it girth} of $G$, denoted by $\textrm{g}(G)$, is the length of the shortest cycle contained in $G$. If the graph does not contain any cycle, its girth is defined to be infinity. All the previous concepts can be defined for directed graphs just considering directed paths.

Let us consider the directed zero divisor graph $\Gamma(\mathbb{Z}_{n}[i,j,k])$. If $n$ is even then $g(\Gamma(\mathbb Z_n[i, j, k]))=2$, since
\begin{equation*}\left[\begin{array}{cc}1&0\\0&0\end{array}\right]\left[\begin{array}{cc}0&0\\0&1
\end{array}\right]=\left[\begin{array}{cc}0&0\\0&1\end{array}\right]\left[\begin{array}{cc}1&0\\0&0\end{array}\right]=
\left[\begin{array}{cc}0&0\\0&0\end{array}\right].\end{equation*}
For $n=2^t$ the ring $\mathbb Z_{2^t}[i, j, k]$ is reversible and consequently $g(\Gamma(\mathbb Z_n[i, j, k]))=2$.

Now, we turn to the undirected case. It is clear that $\overline\Gamma(\mathbb{Z}_{n}[i,j,k])$ is a simple graph; i.e., it does not contain loops and two vertices are not connected by more that one edge. Thus, it follows that $g(\overline\Gamma(\mathbb Z_{2^t}[i, j , k]))\geq3$. We will now see that equality holds.

To compute the girth of $\overline\Gamma(\mathbb Z_n[i, j, k])$ for odd $n$, we recall the following result \cite[Prop. 3.2]{bo}

\begin{Proposition}
Let $R$ be a commutative ring. Then $\textrm{g}(\overline\Gamma(M_n(R)))=3$.
\end{Proposition}

This proposition clearly implies that ${\rm g}(\overline\Gamma(\mathbb Z_{n}[i, j, k])=3$ because, for odd $n$, we have that $\mathbb Z_{n}[i, j, k]\cong M_2(\mathbb{Z}_n)$.

The case $n=2^t$ is analyzed in the following result.

\begin{Proposition}
Let $t\geq 1$. Then ${\rm g}(\overline\Gamma(\mathbb Z_{2^t}[i, j , k]))=3$.
\end{Proposition}
\begin{proof}
If $t=1$, we have the cycle (see the previous remark) $(1+i)$---$(j+k)$---$(1+i+j+k)$---$(1+i)$, for instance.
If $t=2$, we have the cycle $2$---$2i$---$(2+2i)$---$2$, for instance. Finally, if $t>2$ we can consider the cycle $2^{t-1}$---$2$---$2^{t-1}i$---$2^{t-1}$. This proves the result.
\end{proof}

Finally, if we recall isomorphism (\ref{decom}), the previous discussion leads to the following.

\begin{Theorem}
\label{gir}
For every integer $n$, ${\rm g}(\overline{\Gamma}(\mathbb Z_n[i, j, k]))=3$.
\end{Theorem}

A graph $G$ is {\emph complete bipartite} if its vertices can be partitioned into two subsets such that no edge has both endpoints in the same subset, and every possible edge that could connect vertices in different subsets is part of the graph. In \cite[Theorem 17]{ab1} it was proved that the undirected zero divisor graph for the ring of Gaussian integers modulo $n$, $\overline{\Gamma}(\mathbb{Z}_n[i])$, is complete bipartite if and only if $n=p^2$, where $p$ is a rational prime such $p\equiv 1\pmod{4}$ or $n=q_1q_2$, with $q_1,q_2$ rational primes such that $q_1\equiv q_2\equiv 3\pmod{4}$. In our case we have the following.

\begin{Corollary}
The graph $\overline{\Gamma}(\mathbb Z_n[i, j, k])$ is never complete bipartite.
\end{Corollary}
\begin{proof}
The girth of a complete bipartite graph is 4. Since this is not possible due to Theorem \ref{gir}, the result follows.
\end{proof}

\section{The domination number}

A {\emph dominating set} for a graph $G$ is a subset of vertices $D$, such that every vertex not in $D$ is adjacent to at least one member of $D$. The {\em domination number}, denoted by $\gamma(G)$, is the number of vertices in a minimal dominating set.

The problem of determining the dominating number of an arbitrary graph is  NP-complete \cite{ga}. Nevertheless, particular cases have been recently studied. In \cite{ab1}, for instance, the dominating number of the zero divisor graph of the ring of Gaussian integers modulo $n$ was studied. In particular, the authors characterized the values of $n$ for which the domination number of $\Gamma(\mathbb Z_n[i])$ is $1$ or $2$.

This section is devoted to study the domination number of the undirected zero divisor graph $\overline\Gamma(\mathbb Z_{n}[i, j , k])$. The easiest case arises when $n$ is a power of 2.

\begin{Theorem}  \label{dompot2}
The domination number of the undirected zero divisor graph $\overline\Gamma(\mathbb Z_{2^t}[i, j , k])$ is $1$ for every $t\geq 1$.
\end{Theorem}
\begin{proof}
Just observe that $\{2^{t-1}(1+i+j+k)\}$ is a dominating set of the graph $\overline\Gamma(\mathbb Z_{2^s}[i, j , k])$ because $2^{s-1}(1+i+j+k)z=0$ for every non-zero zero divisor $z$.
\end{proof}

The rest of the section will be devoted to study de case when $n$ is an odd prime. In particular we will prove that, for an odd prime number $p$, the domination number of $\overline\Gamma(\mathbb Z_p[i, j , k])$  is $p+1$.

Let $\mathcal{U}(M_2(\mathbb Z_p))$ be the group of units in the matrix ring $M_2(\mathbb Z_p)$. We can consider a natural left group action on the vertices of $\overline\Gamma(M_2(\mathbb Z_p))$ defined by $(U, X)\rightarrow UX$ from $\mathcal{U}(M_2(\mathbb Z_p))\times Z(M_2(\mathbb Z_p))^*$ to $Z(M_2(\mathbb Z_p))^*$.

In the following lemma we characterize the structure of the orbits under the natural left action.

\begin{Lemma}\label{lso}
The distinct orbits of  the regular left action on $Z^*(M_2(\mathbb Z_p))$  are
\begin{equation*}o\left(\left[\begin{array}{cc}1&a\\0&0\end{array}\right]\right),\;\;\;for\;\;some\;\;a\in\mathbb Z_p,
\;\;and \;\;o\left(\left[\begin{array}{cc}0&0\\0&1\end{array}\right]\right).\end{equation*}
\end{Lemma}
\begin{proof}
Let $A$ be an element of $ Z(M_2(\mathbb Z_p))^*$. As is well-known from linear algebra performing a row operation on a matrix, with entries in a field,  is equivalent to multiplying on the left by a suitable invertible matrix. Since by row operations we can reduce any matrix to an upper triangular matrix it follows that there is an invertible matrix $U\in M_2(\mathbb Z_p)$ suct that
\begin{equation*}UA=\left[\begin{array}{cc}\alpha&\beta\\0&\gamma\end{array}\right],\end{equation*}
for some $\alpha, \beta, \gamma\in\mathbb Z_p$.

Assume first that $\alpha\neq 0$. Since $A$ is singular it follows that $\gamma =0$. So, we have
\begin{equation*}\left[\begin{array}{cc}\alpha^{-1}&0\\0&\alpha^{-1}\end{array}\right]\left[\begin{array}{cc}\alpha&\beta\\0&0\end{array}\right]=
\left[\begin{array}{cc}1&\alpha^{-1}\beta\\0&0\end{array}\right],\end{equation*}
and we get the result.

Now assume that  $\alpha=0$.  If $\gamma\neq 0$, then we have
\begin{equation*}\left[\begin{array}{cc}\gamma^{-1}&-\gamma^{-1}\beta\\0&\gamma^{-1}\end{array}\right]\left[\begin{array}{cc}0&\beta\\0&\gamma\end{array}\right]=
\left[\begin{array}{cc}0&0\\0&1\end{array}\right],\end{equation*}
as desired. If $\gamma=0$, since $A$ is not the nul matrix it follows that  $\beta\neq 0$. Hence we have
\begin{equation*}\left[\begin{array}{cc}\beta^{-1}\gamma&\beta^{-1}\\ \beta^{-1}&0\end{array}\right]\left[\begin{array}{cc}0&\beta\\0&\gamma\end{array}\right]=
\left[\begin{array}{cc}0&0\\0&1\end{array}\right],\end{equation*}
and the proof is completed.
\end{proof}

Similarly, if we consider the right regular action of $\mathcal{U}(M_2(\mathbb Z_p))$ on the vertices of $\overline\Gamma(M_2(\mathbb Z_p))$ defined by $(U, X)\rightarrow XU$, we have that the
distinct orbits are
\begin{equation*}o\left(\left[\begin{array}{cc}1&0\\ b&0\end{array}\right]\right),\;\;\;for\;\;some\;\;b\in\mathbb Z_p,
\;\;and \;\;o\left(\left[\begin{array}{cc}0&0\\0&1\end{array}\right]\right).\end{equation*}

On the other hand, for any ring $R$, if $u\in R$ is a unit and $x\in R$ is any element of $R$ we have $ann_r(ux)=ann_r(x)$ and $ann_l(xu)=ann_l(x)$. Consequently, the distinct right
annihilators of a single element in the ring $M_2(\mathbb Z_p)$ are
\begin{equation}\label{eq1} ann_r\left(\left[\begin{array}{cc}1&a\\0&0\end{array}\right]\right)=\left\{\left[\begin{array}{cc}-a z&-a w\\z&w\end{array}\right]\;\;:\;\;z,w\in\mathbb Z_p\right\}\end{equation}
and
\begin{equation}\label{eq2} ann_r\left(\left[\begin{array}{cc}0&0\\0&1\end{array}\right]\right)=\left\{\left[\begin{array}{cc}z&w\\0&0\end{array}\right]\;\;:\;\;z,w\in\mathbb Z_p\right\}\end{equation}

Since any vertex in the graph $\overline\Gamma(M_2(\mathbb Z_p))$ belongs to a right annihilator, if follows that
\begin{equation*}D=\left\{\left[\begin{array}{cc}1&a \\0&0\end{array}\right]\;\;:\;\;a\in\mathbb Z_p\right\}\cup\left\{\left[\begin{array}{cc}0&0\\0&1\end{array}\right]\right\}\end{equation*}
is a dominating set for the graph $\overline\Gamma(M_2(\mathbb Z_p))$. Therefore, $p+1$ is an upper bound for the domination number of the graph
$\overline\Gamma(M_2(\mathbb Z_p))$. We will prove that this upper bound is the exact domination number of $\overline\Gamma(M_2(\mathbb Z_p))$.

Similarly, the distinct left annihilators are
\begin{equation}\label{eq3} ann_l\left(\left[\begin{array}{cc}1&0\\ b&0\end{array}\right]\right)=\left\{\left[\begin{array}{cc}-b x&x \\-b y&y\end{array}\right]\;\;:\;\;x,y\in\mathbb Z_p\right\}\end{equation}
and
\begin{equation}\label{eq4} ann_l\left(\left[\begin{array}{cc}0&0\\0&1\end{array}\right]\right)=\left\{\left[\begin{array}{cc}x&0\\y&0\end{array}\right]\;\;:\;\;x,y\in\mathbb Z_p\right\}\end{equation}

Note that the intersection of two distinct right annihilators (or two left annihilators) contains only the zero element.

\begin{Lemma}\label{li}
The intersection of a right annihilator of type (\ref{eq1}) or (\ref{eq2})  with a left annihilator of type (\ref{eq3}) or (\ref{eq4})   is a set with $p$ elements. \end{Lemma}
\begin{proof}
First note that both a left annihilator and a right annihilator are  subgroups, of order $p^2$, of the additive group $M_2(\mathbb Z_p)$.  Since no right annihilator of type (\ref{eq1}) equals a left annihilator of type (\ref{eq3}) it follows that their intersection is subgroup of order $p$ or else  the trivial subgroup.

The result is clearly true if the right annihilator is of type (\ref{eq2}) or the left annihilator is of type (\ref{eq4}). So, assume that the right annihilator is of type (\ref{eq1}) and the left annihilator is of type (\ref{eq3}).
Note that, the cardinality of the intersection
\begin{equation*}ann_r\left(\left[\begin{array}{cc}1&a\\0&0\end{array}\right]\right)\cap
ann_l\left(\left[\begin{array}{cc}1&0\\ b&0\end{array}\right]\right),\end{equation*}
equals the cardinality of the solution set of the homogeneous system of linear equation
\begin{equation*}a z-b x=0;\;\;x+a w=0;\;\;z+b y=0;\;\;w=y.\end{equation*}
Since the coeficient matrix of this system is singular we get the result.
\end{proof}

\begin{Theorem} \label{domop}
The domination number of the zero divisor graph $\overline\Gamma(\mathbb Z_p[i, j , k])$, where $p$ is an odd prime number, is $p+1$.\end{Theorem}
\begin{proof}
Suppose the theorem were false. Then, we could find a dominating set $D=\{D_1, \ldots, D_k\}$, where $k<p+1$. Hence, there is a right annihilator $ann_r(X)$ of type (\ref{eq1}) or (\ref{eq2}) that is not among the $k$ right annihilators $ann_r(D_1), \ldots, ann_r(D_k)$ and a left annihilator $ann_l(Y)$ that is not among the $k$ left annihilators $ann_l(D_1), \ldots, ann_l(D_k)$. Since all the elements in the right annihilator $ann_r(X)$ are vertices of the graph $\overline\Gamma(\mathbb Z_p[i, j , k])$, it follows that
\begin{equation*}ann_r(X)\subseteq \cap_{i=1}^k ann_l(D_i).\end{equation*}
Therefore, $ann_r(X)\cap ann_l(Y)=\{0\}$, which contradicts lemma \ref{li}.
\end{proof}

As a consequence of the previous result we can easily compute the domination number of $\overline\Gamma(\mathbb Z_n[i, j , k])$ when $n$ is an odd square-free integer.

\begin{Theorem} \label{domosf}
Let $n= p_1\cdots p_k$ with $p_i$ prime for every $i$. Then, the domination number of the zero divisor graph $\overline\Gamma(\mathbb Z_n[i, j , k])$ is $ k+p_1+...+p_k$.
\end{Theorem}
\begin{proof}
Let $C_i :=\{M_{i,1},...,M_{i, 1+p_i}\} $ be the dominating set for the graph $\overline\Gamma(\mathbb Z_{p_i}[i, j , k])$ given by Theore \ref{domop}. Now, it is easy to see that the set
$$   \bigcup_{i=1}^k \{(0,0,...,M_{i,1},...0,0..),...,(0,0,...,M_{i,1+p_i},...0,0..)\} $$
is a minimal dominating set of $\overline\Gamma(\mathbb Z_n[i, j , k])$ and hence the result.
\end{proof}

In a similar way, we can proof the following result.

\begin{Theorem}
Let $n=2^s p_1\cdots p_k$ with $p_i$ is prime for every $i$ and $s>0$. Then, the domination number of the zero divisor graph $\overline\Gamma(\mathbb Z_n[i, j , k])$ is $1+k+p_1+...+p_k$.
\end{Theorem}

We end this section presenting two open problems:
\begin{enumerate}
\item For an odd prime number $p$ and a positive integer $t$, what is the domination number of $\overline\Gamma(\mathbb Z_{p^t}[i, j , k])$?
\item Let $t$ be a positive integer and denote by $\mathbb F_{q}$ the finite field of $q=p^t$ elements. For a positive integer $n>2$, what is the domination number of the zero divisor graph $\overline\Gamma(M_n(\mathbb F_q))$?
\end{enumerate}

\end{document}